\documentclass[11pt,a4paper]{article}
\usepackage{pifont}
\usepackage{bbding}

\usepackage{amssymb}
\usepackage{mathrsfs}
\usepackage{amsfonts}
\usepackage{amsthm,amscd,amsmath}
\usepackage{eufrak}
\usepackage{youngtab}
\usepackage{graphicx}
\usepackage[all]{xy}
\usepackage{indentfirst}

\allowdisplaybreaks

\makeatletter
\def\@biblabel#1{#1}
\makeatother

\newtheorem{theorem}{Theorem}[section]
\newtheorem{lemma}[theorem]{Lemma}

\theoremstyle{definition}
\newtheorem{definition}[theorem]{Definition}

\theoremstyle{proposition}
\newtheorem{proposition}[theorem]{Proposition}

\theoremstyle{remark}

\theoremstyle{example}
\newtheorem{example}[theorem]{Example}

\theoremstyle{corollary}

\date{}

\begin{document}

\title{Decomposition of the Symmetric Powers}

\author{Bin Li \\
\tiny{School of Mathematics and Statistics, Wuhan University, P.R. China} \\
\tiny{libin117@whu.edu.cn}}

\maketitle

\begin{abstract}
A decomposition of any symmetric power of $\Bbb C^2\otimes\Bbb C^2\otimes\Bbb C^2$ into irreducible  $sl_2(\Bbb C)\oplus sl_2(\Bbb C)\oplus sl_2(\Bbb C)$-submodules are presented. Namely, the multiplicities of irreducible summands in the symmetric power are determined.
\end{abstract}


\section{Introduction}

As is known in classical invariant theory, the determinant of a given $n\times n$ square matrix $A$ can be regarded as the simplest non-constant polynomial in the entries of $A$ which is invariant under the action of $SL_n(\Bbb C)\times SL_n(\Bbb C)$\cite{Pr}. As a generalization, Cayley introduced in \cite{Ca} the hyperdeterminant of a multidimensional array. For the simplest case, Cayley's hyperdeterminant, that is, the hyperdeterminant of a $2\times 2\times 2$ array, is known to be a homogeneous polynomial of degree 4 invariant under the action of $SL_2(\Bbb C)\times SL_2(\Bbb C)\times SL_2(\Bbb C)$\cite{GKZ}. Moreover, Gelfand, Kapranov and Zelevinsky pointed out in \cite{GKZ} that any invariant in the polynomial algebra, also called the symmetric algebra of $\Bbb C^2\otimes\Bbb C^2\otimes\Bbb C^2$, is a polynomial in Cayley's hyperdeterminant.

A new proof that Cayley's hyperdeterminant generates all the invariants was shown by Bremner, Bickis and Soltanifar \cite{BBS} in a combinatorial approach via representation theory of Lie algebras. In their proof, as a trivial module over $sl_2(\Bbb C)\oplus sl_2(\Bbb C)\oplus sl_2(\Bbb C)$, Cayley's hyperdeterminant was calculated in a linear algebraic method. An estimation for the multiplicity $x^{m}$ of $V(0)\otimes V(0)\otimes V(0)$ in the $m^{th}$-symmetric power $S^{m}(\Bbb C^2\otimes\Bbb C^2\otimes\Bbb C^2)$ was given in \cite{BBS} which compels $x^{m}$ to be 1 if $m$ is a multiple of 4, and 0 otherwise. Inspired by their work, we determine the multiplicity $x^{m}_{n_1,n_2,n_3}$ of $V(n_1)\otimes V(n_2)\otimes V(n_3)$ in the $m^{th}$-symmetric power for any $m, n_1, n_2, n_3\in\Bbb Z^{+}$ by studying the character of $S^{m}(\Bbb C^2\otimes\Bbb C^2\otimes\Bbb C^2)$. Since every irreducible $sl_2(\Bbb C)\oplus sl_2(\Bbb C)\oplus sl_2(\Bbb C)$-module is isomorphic to $V(n_1)\otimes V(n_2)\otimes V(n_3)$ for some $n_1, n_2, n_3\in\Bbb Z^{+}$ and the symmetric power is completely reducible, we decompose  $S^{m}(\Bbb C^2\otimes\Bbb C^2\otimes\Bbb C^2)$ into a direct sum of irreducible $sl_2(\Bbb C)\oplus sl_2(\Bbb C)\oplus sl_2(\Bbb C)$-submodules.

The paper is organized as follows: in section 2, we briefly recall some  basic definitions and in section 3, the multiplicities of the irreducibles in the symmetric power are determined by the dimensions of the weight spaces. In section 4, we present the formulas for the dimensions of the weight spaces in the symmetric power.

\section{Preliminaries}

We denote by $sl_2(\Bbb C)$ the simple Lie algebra consisting of all the $2\times 2$ matrices
\begin{gather*}
\begin{pmatrix} a & b \\ c & -a\end{pmatrix}
\end{gather*}
over the complex field $\Bbb C$. Clearly, $sl_2(\Bbb C)$ is 3-dimensional with a $\Bbb C$-basis
\begin{gather*}
E=\begin{pmatrix} 0 & 1 \\ 0 & 0\end{pmatrix},\quad H=\begin{pmatrix} 1 & 0 \\ 0 & -1\end{pmatrix}, \quad  F=\begin{pmatrix} 0 & 0 \\ 1 & 0\end{pmatrix}.
\end{gather*}
The Cartan subalgebra of $sl_2(\Bbb C)$ is spanned over $\Bbb C$ by $H$, denoted by $\mathfrak h$. Given a Lie algebra $\mathfrak g$ and its Cartan subalgebra $\mathfrak h$, we say a $\mathfrak g$-module $M$ has a weight space decomposition if $M=\bigoplus M_{\lambda}$, where $\lambda\in\mathfrak h^{\ast}$ and $M_{\lambda}=\{m\in M\ |\ h\cdot m= \lambda(h)m\ \textrm{for all}\ h\in\mathfrak h \}$, called weight space of weight $\lambda$.

 Let $\Bbb C^2$ be the set of all 2-dimensional complex column vectors which admits a natural $sl_2(\Bbb C)$-action through matrix multiplication. Take $\alpha, \omega\in\mathfrak h^{\ast}$ such that
 \[\alpha(H)=2,\ \omega(H)=1.\] Note that $\alpha$ and $\omega$ are the unique simple root and the unique fundamental weight of $sl_2(\Bbb C)$ respectively. It is easy to see $x_0=(1,0)^\textrm{T}$ and $x_1=(0,1)^\textrm{T}$ form a basis of $\Bbb C^2$ with weights $\omega$ and $-\omega$. It is well-known that every finite dimensional irreducible $sl_2$-module has a weight space decomposition and is determined uniquely by its highest weight. Moreover, for every non-negative integer $n$, there is an $(n+1)$-dimensional irreducible $sl_2(\Bbb C)$-module of highest weight $n\omega$, or simply $n$, denoted by $V(n)$. Actually $\Bbb C^2$ above is $V(1)$.

 For a given Lie algebra $\mathfrak g$ and a $\mathfrak g$-module $M$, we write $T(M)$ for the tensor algebra of $M$. Note that \[T(M)=\bigoplus_{m\geqslant 0 }T^{m}(M),\] where $T^{m}(M)=\underbrace{M\otimes \cdots \otimes M}_{m}$ is called the $m^{th}$-tensor power.
The multiplication in $T(M)$ is defined as
\[(x_1\otimes\cdots\otimes x_m)\cdot(y_1\otimes\cdot\otimes y_n)=x_1\otimes \cdots\otimes x_m\otimes y_1\otimes \cdots\otimes y_n.\]
We define the symmetric algebra of $M$ as the quotient of $T(M)$ by the ideal generated by
\[x\otimes y-y\otimes x\in T^2(M) =M\otimes M\]
for all $x, y\in M$, which we denote by $S(M)$. One has $S(M)=\bigoplus_{m\geqslant 0}S^{m}(M)$, where $S^{m}(M)$ is called the $m^{th}$-symmetric power.
Note that every tensor power admits a $\mathfrak g$-action as follows
\[g\cdot (x_1\otimes \cdots\otimes x_m)=\sum_{i=1}^{m}x_1\otimes\cdots\otimes (g\cdot x_i)\otimes \cdots \otimes x_m.\]
Since the ideal defined above is also a $\mathfrak g$-module, so is the symmetric power $S^{m}(M)$.

   From now on, we will concentrate on the Lie algebra $sl_2(\Bbb C)\oplus sl_2(\Bbb C)\oplus sl_2(\Bbb C)$, which is also denoted by $sl_2(\Bbb C)^{\oplus 3}$ for simplicity. Let $\alpha_i$ and $\omega_i$ be the simple root and the fundamental weight of the $i^{th}$-copy of $sl_2(\Bbb C)$ in $sl_2(\Bbb C)^{\oplus 3}$ and let $\{E_i, H_i, F_i\}$ be the basis of the copy as above. The Cartan subalgebra $\mathfrak h$ of $sl_2(\Bbb C)^{\oplus 3}$ is then defined as $\mathfrak h=\bigoplus_{i=1}^{3}\Bbb CH_i.$  A $sl_2(\Bbb C)^{\oplus 3}$-module $M$ which has a weight space decomposition can be written as
 \[M=\bigoplus M_{(\lambda_1,\lambda_2,\lambda_3)},\] where $(\lambda_1,\lambda_2,\lambda_3)\in\mathfrak h^{\ast}$ acts on $\mathfrak h$ as
 \[(\lambda_1,\lambda_2,\lambda_3)(a_1H_1+a_2H_1+a_3H_1)=\sum_{i=1}^3a_i\lambda_i(H_i), (a_i\in\Bbb C).\] We also write $(\lambda_1(H_1), \lambda_2(H_2), \lambda_3(H_3))\in\Bbb C^{3}$ for the weight.
 Note that  $\Bbb C^2\otimes\Bbb C^2\otimes\Bbb C^2$ can be viewed as a $sl_2(\Bbb C)^{\oplus 3}$-module by
 \[ (x, y, z)\cdot(v_1\otimes v_2\otimes v_3)=(x\cdot v_1)\otimes v_2\otimes v_3+v_1\otimes (y\cdot v_2)\otimes v_3+v_1\otimes v_2\otimes(z\cdot v_3),\] and $x_{i,j,k}\triangleq x_i\otimes x_j\otimes x_k$ ($i, j, k\in\{0,1\}$) spans the 1-dimensional weight space $\Bbb C^2\otimes\Bbb C^2\otimes\Bbb C^2_{(1-2i,1-2j,1-2k)}$.

\section{Multiplicities}

Since $\Bbb C^2\otimes\Bbb C^2\otimes\Bbb C^2$ is 8-dimensional, the $m^{th}$-symmetric power of $\Bbb C^2\otimes\Bbb C^2\otimes\Bbb C^2$ is finite dimensional for every $m\in\Bbb Z^{+}$. Indeed, all the ordered monomials of degree $m$ in 8 variables $x_{i,j,k}$ ($i, j, k\in\{0,1\}$) form a basis for $S^m(\Bbb C^2\otimes\Bbb C^2\otimes\Bbb C^2)$ and obviously there are only finitely many such monomials. Since known as Weyl's theorem\cite{Hu}, every finite dimensional modules over a complex semisimple Lie algebra is completely reducible, especially the $m^{th}$-symmetric power $S^m(\Bbb C^2\otimes\Bbb C^2\otimes\Bbb C^2)$ decomposes into a direct sum of irreducible $sl_2(\Bbb C)^{\oplus 3}$-submodules.
\begin{lemma}
Every finite dimensional irreducible $sl_2(\Bbb C)^{\oplus 3}$-module is isomorphic to a tensor product $V(n_1)\otimes V(n_2)\otimes V(n_3)$ for some $n_i\in\Bbb Z^+$, $i=1, 2, 3$.
\end{lemma}
It follows immediately from this lemma that
\[S^m(\Bbb C^2\otimes\Bbb C^2\otimes\Bbb C^2)=\bigoplus_{n_1,n_2,n_3\in\Bbb Z^{+}}V(n_1)\otimes V(n_2)\otimes V(n_3)^{\oplus x_{n_1,n_2,n_3}^m}.\]
More generally, instead of $x_{n_1,n_2,n_3}^m$, we will determine the multiplicity of $V(n_1)\otimes V(n_2)\otimes V(n_3)$ in any finite dimensional $sl_2(\Bbb C)^{\oplus 3}$-module $M$. In order to do that, one needs to study the character of a $\mathfrak g$-module.
\begin{definition}
The character of a $\mathfrak g$-module $M=\bigoplus_{\lambda\in\mathfrak h^{\ast}}M_{\lambda}$ is the formal (possibly infinite) sum
\[ch(M)\triangleq\sum_{\lambda\in\mathfrak h^{\ast}}dim(M_{\lambda})e^{\lambda}.\]
\end{definition}
For instance, the character of the $sl_2(\Bbb C)$-module $V(n)$ can be written as
\[ch(V(n))=e^{n}+e^{n-2}+\cdots + e^{2-n}+e^{-n}=\sum_{i=0}^{n}e^{n-2i}.\]
It is clear that the character of the irreducible $sl_2(\Bbb C)^{\oplus 3}$-module $V(n_1)\otimes V(n_2)\otimes V(n_3)$
is
\[ch(V(n_1)\otimes V(n_2)\otimes V(n_3))=\sum_{1\leqslant i_j\leqslant n_j}e^{(n_1-2i_1,n_2-2i_2,n_3-2i_3)}.\]
Note that every weight space in $V(n_1)\otimes V(n_2)\otimes V(n_3)$ is 1-dimensional. Since $ch(M_1\oplus M_2)=ch(M_1)+ch(M_2)$, to decompose a finite dimensional $sl_2(\Bbb C)^{\oplus 3}$-module $M$ into irreducible submodules, one only needs to write the character of $M$ as the sum of $ch(V(n_1)\otimes V(n_2)\otimes V(n_3))$'s. Indeed, for any finite dimensional $sl_{2}(\Bbb C)^{\oplus 3}$-module $M$, the decomposition of $M$ proceeds as follows.

We define a partial order $\preccurlyeq $ on the weight set of $M$ as
\[(m_1, m_2, m_3)\preccurlyeq (l_1, l_2, l_3)\ \textrm{if}\ m_i\leqslant l_i\ \textrm{and}\ 2\mid (l_i-m_i)\ \textrm{for all}\ i=1, 2, 3.\]
The formal sum $\sum_{j_1,j_2,j_3\in\Bbb C}n_{j_1,j_2,j_3}e^{(j_1, j_2, j_3)}$ is defined to act on $(i_1, i_2, i_3)\in\Bbb C^3$ as
\[(\sum_{j_1,j_2,j_3\in\Bbb C}n_{j_1,j_2,j_3}e^{(j_1, j_2, j_3)})((i_1, i_2, i_3))=n_{i_1, i_2, i_3}\in \Bbb Z^+.\]
First take the formal sum $f_1=ch(M)$. Choose a maximal weight $(n_{11}, n_{12}, n_{13})$ with respect to $\preccurlyeq $ among all $(m_1, m_2, m_3)$'s such that $f_1((m_1, m_2, m_3))\neq 0$. Note that $n_{1i}\in\Bbb Z^+$. Next, set
$f_2=ch(M)-ch(V(n_{11})\otimes V(n_{12})\otimes V(n_{13}))$ and take a maximal weight $(n_{21}, n_{22}, n_{23})$ among all $(l_1, l_2, l_3)$'s such that $f_2((l_1, l_2, l_3))\neq 0$. Then we define \[f_3=ch(M)-\sum_{j=1}^{2}ch(V(n_{j1})\otimes V(n_{j2})\otimes V(n_{j3}))\]
and similar operation continues until  $f_{s+1}=0$ for some $s\in\Bbb Z^{+}$. Since $ch(M)$ is a finite sum, the algorithm stops after finitely many steps.
Finally one has
\[M=\bigoplus_{i=1}^{s}V(n_{i1})\otimes V(n_{i2})\otimes V(n_{i3}).\]
The following theorem shows another method which is more effective to calculate the multiplicities of the irreducible summands in $M$. See also Theorem 36 in \cite{BBS} for the case $M=S^m(\Bbb C^2\otimes\Bbb C^2\otimes\Bbb C^2)$ and $n_1=n_2=n_3=0$.
\begin{theorem}
For a finite dimensional $sl_2(\Bbb C)^{\oplus 3}$-module $M$, the multiplicity $x_{n_1,n_2,n_3}$ of $V(n_1)\otimes V(n_2)\otimes V(n_3)$ in $M$ is
\begin{equation*}
\begin{split}
x_{n_1,n_2,n_3}=&dimM_{(n_1,n_2,n_3)}-dimM_{(n_1+2,n_2,n_3)}-dimM_{(n_1,n_2+2,n_3)}\\
                &-dimM_{(n_1,n_2,n_3+2)}+dimM_{(n_1+2,n_2+2,n_3)}+dimM_{(n_1+2,n_2,n_3+2)}\\
                &+dimM_{(n_1,n_2+2,n_3+2)}-dimM_{(n_1+2,n_2+2,n_3+2)}.\\
\end{split}
\end{equation*}
\end{theorem}
\begin{proof} We assume that
\[M=\bigoplus_{n_1,n_2,n_3\in\Bbb Z^+}V(n_1)\otimes V(n_2)\otimes V(n_3)^{\oplus x_{n_1,n_2,n_3}}.\]
It is easy to check that for $m_i\in\Bbb Z^+$ ($i=1, 2, 3$),
\[(m_1, m_2, m_3)\preccurlyeq (l_1, l_2, l_3)\ \textrm{iff}\ V(l_1)\otimes V(l_2)\otimes V(l_3)_{(m_1, m_2, m_3)}\neq 0.\] As mentioned above,
$V(l_1)\otimes V(l_2)\otimes V(l_3)_{(m_1, m_2, m_3)}$ is 1-dimensional if it is non-zero.
It follows that
\[dimM_{(m_1,m_2,m_3)}=\sum_{(m_1, m_2, m_3)\preccurlyeq (l_1,l_2,l_3)}x_{l_1, l_2, l_3}\]
 for $m_i\in\Bbb Z^+$ ($i=1, 2, 3$).
 For given $n_i\in\Bbb Z^+$ ($i=1, 2, 3$), we divide the sum $dimM_{(n_1,n_2,n_3)}=\sum_{(n_1, n_2, n_3)\preccurlyeq (l_1,l_2,l_3)}x_{l_1, l_2, l_3}$ into 8 parts as the following.
\begin{align}
dimM_{(n_1,n_2,n_3)}=
&x_{n_1,n_2,n_3}\\
&+\sum_{(n_1, n_2, n_3+2)\preccurlyeq (n_1,n_2,l_3)}x_{n_1, n_2, l_3}\\
&+\sum_{(n_1, n_2+2, n_3)\preccurlyeq (n_1,l_2,n_3)}x_{n_1, l_2, n_3}\\
&+\sum_{(n_1+2, n_2, n_3)\preccurlyeq (l_1,n_2,n_3)}x_{l_1, n_2, n_3}\\
&+\sum_{(n_1, n_2+2, n_3+2)\preccurlyeq (n_1,l_2,l_3)}x_{n_1, l_2, l_3}\\
&+\sum_{(n_1+2, n_2, n_3+2)\preccurlyeq (l_1,n_2,l_3)}x_{l_1, n_2, l_3}\\
&+\sum_{(n_1+2, n_2+2, n_3)\preccurlyeq (l_1,l_2,n_3)}x_{l_1, l_2, n_3}\\
&+\sum_{(n_1+2, n_2+2, n_3+2)\preccurlyeq (l_1,l_2,l_3)}x_{l_1, l_2, l_3}
\end{align}
Check that
\begin{align*}
dimM_{(n_1+2,n_2,n_3)}&= (4)+(6)+(7)+(8),\\
dimM_{(n_1,n_2+2,n_3)}&= (3)+(5)+(7)+(8),\\
dimM_{(n_1,n_2,n_3+2)}&= (2)+(5)+(6)+(8),\\
dimM_{(n_1+2,n_2+2,n_3)}&=  (7)+(8),\\
dimM_{(n_1+2,n_2,n_3+2)}&= (6)+(8),\\
dimM_{(n_1,n_2+2,n_3+2)}&= (5)+(8),\\
dimM_{(n_1+2,n_2+2,n_3+2)}&= (8).\\
\end{align*}
One can complete the proof by a direct calculation.
\end{proof}
Applying this theorem, we get the multiplicity of of $V(n_1)\otimes V(n_2)\otimes V(n_3)$ in $S^m(\Bbb C^2\otimes\Bbb C^2\otimes\Bbb C^2)$ as the following
\begin{equation*}
\begin{split}
x_{n_1,n_2,n_3}^{m}=&dimV^m_{(n_1,n_2,n_3)}-dimV^m_{(n_1+2,n_2,n_3)}-dimV^m_{(n_1,n_2+2,n_3)}\\
                    &-dimV^m_{(n_1,n_2,n_3+2)}+dimV^m_{(n_1+2,n_2+2,n_3)}+dimV^m_{(n_1+2,n_2,n_3+2)}\\
                    &+dimV^m_{(n_1,n_2+2,n_3+2)}-dimV^m_{(n_1+2,n_2+2,n_3+2)},\\
\end{split}
\end{equation*}
where $V^{m}=S^m(\Bbb C^2\otimes\Bbb C^2\otimes\Bbb C^2)$.

\section{Dimension formulas}

Fix $m\in\Bbb Z^+$. To obtain the multiplicity of $V(n_1)\otimes V(n_2)\otimes V(n_3)$  in $S^m(\Bbb C^2\otimes\Bbb C^2\otimes\Bbb C^2)$, it only requires us to compute the dimensions of weight spaces in $S^m(\Bbb C^2\otimes\Bbb C^2\otimes\Bbb C^2)$. In this section, the dimension formulas for the weight spaces will be
given.

Recall that all the ordered monomials of the form
\begin{equation}
x_{0,0,0}^{a_{0,0,0}}x_{0,0,1}^{a_{0,0,1}}x_{0,1,0}^{a_{0,1,0}}x_{0,1,1}^{a_{0,1,1}}x_{1,0,0}^{a_{1,0,0}}
x_{1,0,1}^{a_{1,0,1}}x_{1,1,0}^{a_{1,1,0}}x_{1,1,1}^{a_{1,1,1}},
\end{equation}
where
\begin{equation}
\sum_{i,j,l\in\{0,1\}}a_{i,j,l}=m
\end{equation}
and $a_{i,j,l}\in\Bbb Z^+$, form a basis for $S^m(\Bbb C^2\otimes\Bbb C^2\otimes\Bbb C^2)$.
One can check that the weight of (9) is
\[(m-2k, m-2r, m-2n),\]
where
\begin{align}
k&=a_{1,0,0}+a_{1,1,0}+a_{1,0,1}+a_{1,1,1},\\
r&=a_{0,1,0}+a_{0,1,1}+a_{1,1,0}+a_{1,1,1},\\
n&=a_{0,0,1}+a_{0,1,1}+a_{1,0,1}+a_{1,1,1}.
\end{align}
For fixed $k, r, n\in\Bbb Z^+$, the dimension of \[S^m(\Bbb C^2\otimes\Bbb C^2\otimes\Bbb C^2)_{(m-2k, m-2r, m-2n)},\]
denoted by $C^{m}_{k,r,n}$,
equals the number of choices of $a_{i,j,l}\in\Bbb Z^+$ satisfying (10), (11), (12), (13). Since obviously
\[C^{m}_{k_1,k_2,k_3}=C^{m}_{k_{\sigma(1)},k_{\sigma(2)},k_{\sigma(3)}}\]
for any permutation $\sigma$ of $\{1, 2, 3\}$, we assume that
\begin{equation}
k\geqslant r\geqslant n\geqslant 0.
\end{equation}
It can be easily seen from the structure of $V(n_1)\otimes V(n_2)\otimes V(n_3)$ that
\[dimV(n_1)\otimes V(n_2)\otimes V(n_3)_{(k_1,k_2,k_3)}=dimV(n_1)\otimes V(n_2)\otimes V(n_3)_{(|k_1|,|k_2|,|k_3|)},\]
where $|k_i|$ is the absolute value of $k_i\in\Bbb Z$. Hence for any finite dimensional $sl_2(\Bbb C)^{\oplus 3}$-module $M$, in particular, for
$M=S^m(\Bbb C^2\otimes\Bbb C^2\otimes\Bbb C^2)$, one has
\[dimM_{(k_1,k_2,k_3)}=dimM_{(|k_1|,|k_2|,|k_3|)}.\]
This allows us to furtherly assume that
\begin{equation}
m-2k\geqslant 0,\ m-2r\geqslant 0,\ m-2n\geqslant 0.
\end{equation}
Combining (14), we have the following assumption
\begin{equation}
\frac{m}{2}\geqslant k\geqslant r\geqslant n\geqslant 0.
\end{equation}
For $r_i\in\Bbb Z^+$ with $r_1\geqslant r_2$, $r_1\geqslant r_3$, we define $C^{r_1}_{r_2, r_3}$ to be the number of all $2\times 2$ matrices
\begin{gather*}
\begin{pmatrix} a_{11} & a_{12} \\ a_{21} & a_{22}\end{pmatrix}
\end{gather*}
such that
\[a_{11}+a_{12}+a_{21}+a_{22}=r_1,\]
\[a_{21}+a_{22}=r_2,\ a_{12}+a_{22}=r_3,\ a_{ij}\in\Bbb Z^+.\]
It is proved in \cite{BBS} that
\begin{equation}
C^{r_1}_{r_2, r_3}=min\{r_2, r_3, r_1-r_2, r_2-r_3\}+1.
\end{equation}
Set $a=a_{0,1,0}+a_{0,1,1}$ and $b=a_{0,0,1}+a_{0,1,1}$. See that
\[m-k\geqslant a\geqslant 0,\ \ m-k\geqslant b\geqslant 0,\]
\[k\geqslant r-a\geqslant 0,\ \ k\geqslant n-b\geqslant 0.\]
    Let $S$ be the set of all pairs of $2\times 2$ matrices $(A_0, A_1)$, where
\begin{gather*}
A_1=\begin{pmatrix} a_{0,0,0} & a_{0,0,1} \\ a_{0,1,0} & a_{0,1,1}\end{pmatrix},\quad A_2=
\begin{pmatrix} a_{1,0,0} & a_{1,0,1} \\ a_{1,1,0} & a_{1,1,1}\end{pmatrix}
\end{gather*}
satisfy
\begin{align*}
a_{0,0,0}+a_{0,0,1}+a_{0,1,0}+a_{0,1,1}&=m-k,\\
a_{0,1,0}+a_{0,1,1}&=a,\\
a_{0,0,1}+a_{0,1,1}&=b,\\
a_{1,0,0}+a_{1,0,1}+a_{1,1,0}+a_{1,1,1}&=k,\\
a_{0,1,0}+a_{0,1,1}&=r-a,\\
a_{0,0,1}+a_{0,1,1}&=n-b,\\
0\leqslant a\leqslant r,\ \ 0\leqslant b\leqslant n,\ \ a_{i,j,l}&\in\Bbb Z^+.
\end{align*}
Note that $S$ provides all choices for $a_{i,j,l}\in\Bbb Z^+$ which satisfy (10), (11), (12), (13). It follows that
\begin{equation}
C^{m}_{k,r,n}=\sum_{0\leqslant a\leqslant r,0\leqslant b\leqslant n}C^{m-k}_{a,b}C^{k}_{r-a,n-b}.
\end{equation}
Analyzing over one hundred cases and combining some of them, we obtain 6 different expressions of the formula for $C^{m}_{k,r,n}$ which are listed
in the following proposition.

\begin{proposition}
For $m, k, r, n\in Z^+$ with $\frac{m}{2}\geqslant k\geqslant r\geqslant n\geqslant 0$,
\begin{itemize}
\item[(I)] When $r+n\leqslant k$,
    \begin{equation*}
     C^{m}_{k,r,n}=1+\frac{4 n}{3}+\frac{n^2}{12}-\frac{n^3}{3}-\frac{n^4}{12}+r+\frac{11 n r}{6}+n^2 r+\frac{n^3 r}{6};
    \end{equation*}
\item[(II)] When $k<r+n<m-k$,
     \begin{itemize}
        \item[(II.1)] if $2\mid (r+n-k)$,
            \begin{align*}
                 C^{m}_{k,r,n}=&1+\frac{k}{3}-\frac{5 k^2}{12}+\frac{k^3}{6}-\frac{k^4}{48}+n+\frac{5 k n}{6}-\frac{k^2 n}{2}+\frac{k^3 n}{12}\ \ \ \ \ \ \ \ \ \ \ \\
                 &-\frac{n^2}{3}+\frac{k n^2}{2}-\frac{k^2 n^2}{8}-\frac{n^3}{2}+\frac{k n^3}{12}-\frac{5 n^4}{48}+\frac{2 r}{3}\\
                 &+\frac{5 k r}{6}-\frac{k^2 r}{2}+\frac{k^3 r}{12}+n r+k n r-\frac{1}{4} k^2 n r+\frac{n^2 r}{2}\\
                 &+\frac{1}{4} k n^2 r+\frac{n^3 r}{12}-\frac{5 r^2}{12}+\frac{k r^2}{2}-\frac{k^2 r^2}{8}-\frac{n r^2}{2}\\
                 &+\frac{1}{4} k n r^2-\frac{n^2 r^2}{8}-\frac{r^3}{6}+\frac{k r^3}{12}-\frac{n r^3}{12}-\frac{r^4}{48};
            \end{align*}
        \item[(II.2)] if $2\nmid (r+n-k)$,
           \begin{align*}
                 C^{m}_{k,r,n}=&\frac{15}{16}+\frac{k}{3}-\frac{5 k^2}{12}+\frac{k^3}{6}-\frac{k^4}{48}+n+\frac{5 k n}{6}-\frac{k^2 n}{2}+\frac{k^3 n}{12}\ \  \ \ \ \ \ \ \\
                 &-\frac{n^2}{3}+\frac{k n^2}{2}-
                 \frac{k^2 n^2}{8}-\frac{n^3}{2}+\frac{k n^3}{12}-\frac{5 n^4}{48}+\frac{2 r}{3}\\
                 &+\frac{5 k r}{6}-\frac{k^2 r}{2}+\frac{k^3 r}{12}+n r+k n r-
                 \frac{1}{4} k^2 n r+\frac{n^2 r}{2}\\
                 &+\frac{1}{4} k n^2 r+\frac{n^3 r}{12}-\frac{5 r^2}{12}+\frac{k r^2}{2}-\frac{k^2 r^2}{8}-\frac{n r^2}{2}\\
                 &+\frac{1}{4} k n r^2-\frac{n^2 r^2}{8}-\frac{r^3}{6}+\frac{k r^3}{12}-\frac{n r^3}{12}-\frac{r^4}{48};
           \end{align*}
    \end{itemize}
\item[(III)] When $m-k\leqslant r+n$,
            \begin{itemize}
               \item[(III.1)] if $2\mid m$, $2\mid(r+n-k)$,
                       \begin{align*}
                          C^{m}_{k,r,n}=&1-\frac{5 k^2}{6}-\frac{k^4}{24}+\frac{m}{3}+\frac{5 k m}{6}+\frac{k^2 m}{2}+\frac{k^3 m}{12}-\frac{5 m^2}{12}\ \ \ \ \ \ \ \ \  \\
                          &-\frac{k m^2}{2}-\frac{k^2 m^2}{8}+
                          \frac{m^3}{6}+\frac{k m^3}{12}-\frac{m^4}{48}+\frac{2 n}{3}-k^2 n\\
                          &+\frac{5 m n}{6}+k m n+\frac{1}{4} k^2 m n-\frac{m^2 n}{2}-\frac{1}{4} k m^2 n+
                          \frac{m^3 n}{12}\\
                          &-\frac{3 n^2}{4}-\frac{k^2 n^2}{4}+\frac{m n^2}{2}+\frac{1}{4} k m n^2-\frac{m^2 n^2}{8}-\frac{2 n^3}{3}\\
                          &+\frac{m n^3}{12}-\frac{n^4}{8}+\frac{r}{3}-
                          k^2 r+\frac{5 m r}{6}+k m r+\frac{1}{4} k^2 m r\\
                          &-\frac{m^2 r}{2}-\frac{1}{4} k m^2 r+\frac{m^3 r}{12}+\frac{n r}{6}-\frac{1}{2} k^2 n r+
                          m n r\\
                          &+\frac{1}{2} k m n r-\frac{1}{4} m^2 n r+\frac{1}{4} m n^2 r-\frac{5 r^2}{6}-\frac{k^2 r^2}{4}+\frac{m r^2}{2}\\
                          &+\frac{1}{4} k m r^2-
                          \frac{m^2 r^2}{8}-n r^2+\frac{1}{4} m n r^2-\frac{n^2 r^2}{4}-\frac{r^3}{3}\\
                          &+\frac{m r^3}{12}-\frac{n r^3}{6}-\frac{r^4}{24};
                      \end{align*}
               \item[(III.2)] if $2\mid m$, $2\nmid(r+n-k)$,
                        \begin{align*}
                           C^{m}_{k,r,n}=&\frac{7}{8}-\frac{5 k^2}{6}-\frac{k^4}{24}+\frac{m}{3}+\frac{5 k m}{6}+\frac{k^2 m}{2}+\frac{k^3 m}{12}-\frac{5 m^2}{12}\ \ \ \ \ \ \ \ \\
                           &-\frac{k m^2}{2}-\frac{k^2 m^2}{8}+\frac{m^3}{6}+\frac{k m^3}{12}-\frac{m^4}{48}+\frac{2 n}{3}-k^2 n\\
                           &+\frac{5 m n}{6}+k m n+\frac{1}{4} k^2 m n-\frac{m^2 n}{2}-\frac{1}{4} k m^2 n+\frac{m^3 n}{12}\\
                           &-\frac{3 n^2}{4}-\frac{k^2 n^2}{4}+\frac{m n^2}{2}+\frac{1}{4} k m n^2-\frac{m^2 n^2}{8}-\frac{2 n^3}{3}\\
                           &+\frac{m n^3}{12}-\frac{n^4}{8}+\frac{r}{3}-k^2 r+\frac{5 m r}{6}+k m r+\frac{1}{4} k^2 m r\\
                           &-\frac{m^2 r}{2}-\frac{1}{4} k m^2 r+\frac{m^3 r}{12}+\frac{n r}{6}-\frac{1}{2} k^2 n r+m n r\\
                           &+\frac{1}{2} k m n r-\frac{1}{4} m^2 n r+\frac{1}{4} m n^2 r-\frac{5 r^2}{6}-\frac{k^2 r^2}{4}+\frac{m r^2}{2}\\
                           &+\frac{1}{4} k m r^2-\frac{m^2 r^2}{8}-n r^2+\frac{1}{4} m n r^2-\frac{n^2 r^2}{4}-\frac{r^3}{3}\\
                           &+\frac{m r^3}{12}-\frac{n r^3}{6}-\frac{r^4}{24};
                       \end{align*}
               \item[(III.3)] if $2\nmid m$,
                        \begin{align*}
                           C^{m}_{k,r,n}=&\frac{15}{16}-\frac{5 k^2}{6}-\frac{k^4}{24}+\frac{m}{3}+\frac{5 k m}{6}+\frac{k^2 m}{2}+\frac{k^3 m}{12}-\frac{5 m^2}{12}\ \ \ \ \ \ \ \ \\
                           &-\frac{k m^2}{2}-\frac{k^2 m^2}{8}+\frac{m^3}{6}+\frac{k m^3}{12}-\frac{m^4}{48}+\frac{2 n}{3}-k^2 n\\
                           &+\frac{5 m n}{6}+k m n+\frac{1}{4} k^2 m n-\frac{m^2 n}{2}-\frac{1}{4} k m^2 n+\frac{m^3 n}{12}\\
                           &-\frac{3 n^2}{4}-\frac{k^2 n^2}{4}+\frac{m n^2}{2}+\frac{1}{4} k m n^2-\frac{m^2 n^2}{8}-\frac{2 n^3}{3}\\
                           &+\frac{m n^3}{12}-\frac{n^4}{8}+\frac{r}{3}-k^2 r+\frac{5 m r}{6}+k m r+\frac{1}{4} k^2 m r\\
                           &-\frac{m^2 r}{2}-\frac{1}{4} k m^2 r+\frac{m^3 r}{12}+\frac{n r}{6}-\frac{1}{2} k^2 n r+m n r\\
                           &+\frac{1}{2} k m n r-\frac{1}{4} m^2 n r+\frac{1}{4} m n^2 r-\frac{5 r^2}{6}-\frac{k^2 r^2}{4}+\frac{m r^2}{2}\\
                           &+\frac{1}{4} k m r^2-\frac{m^2 r^2}{8}-n r^2+\frac{1}{4} m n r^2-\frac{n^2 r^2}{4}-\frac{r^3}{3}\\
                           &+\frac{m r^3}{12}-\frac{n r^3}{6}-\frac{r^4}{24};
                        \end{align*}
            \end{itemize}
\end{itemize}
\end{proposition}

\begin{proof} We only show $(I)$, the simplest case, and the others can be similarly checked.
For $m, k, r, n, a, b\in\Bbb Z^+$ with $n\leqslant r\leqslant k\leqslant \frac{m}{2}$, $0\leqslant a\leqslant r$ and $0\leqslant b\leqslant n$, it follows from (17) that
\begin{align}
C^{m-k}_{a,b}&=
\begin{cases}
a+1& \text{if $m-k\geqslant a+b$ and $a\leqslant b$},\\
b+1& \text{if $m-k\geqslant a+b$ and $a> b$},\\
m-k-b+1& \text{if $m-k< a+b$ and $a\leqslant b$},\\
m-k-a+1& \text{if $m-k< a+b$ and $a> b$},
\end{cases}\\
C^{k}_{r-a,n-b}&=
\begin{cases}
k-n+b+1& \text{if $r+n-k\geqslant a+b$ and $r-a\leqslant n-b$},\\
k-r+a+1& \text{if $r+n-k\geqslant a+b$ and $r-a> n-b$},\\
r-a+1& \text{if $r+n-k< a+b$ and $r-a\leqslant n-b$},\\
n-b+1& \text{if $r+n-k< a+b$ and $r-a>n-b$}.
\end{cases}
\end{align}
If $r+n\leqslant k$, one has $r+n-k\leqslant 0\leqslant a+b$ . Moreover, the assumption $\frac{m}{2}\geqslant k$ implies
$m-k-r-n=m-2k-(r+n-k)\geqslant 0$.
Hence $m-k\geqslant r+n\geqslant a+b$. We have, in case $(I)$,
\begin{align}
C^{m-k}_{a,b}&=
\begin{cases}
a+1& \text{if $a\leqslant b$},\\
b+1& \text{if $a> b$},
\end{cases}\\
C^{k}_{r-a,n-b}&=
\begin{cases}
r-a+1& \text{if $r-a\leqslant n-b$},\\
n-b+1& \text{if $r-a>n-b$}.
\end{cases}
\end{align}
Let $C=(c_{ij})$ and $D=(d_{ij})$ be $(r+1)\times (n+1)$ matrices such that
\[c_{ij}=C^{m-k}_{i-1, j-1},\quad d_{ij}=C^{k}_{r+1-i,n+1-j}.\]
One can see that
\begin{equation*}
C^{m}_{k,r,n}=\sum_{1\leqslant i\leqslant r+1,1\leqslant j\leqslant n+1}c_{ij}d_{ij}=\sum_{1\leqslant j\leqslant n+1}C_j\cdot D_j,
\end{equation*}
where $C_j$ and $D_j$ are the the $j^{th}$-column of $C$ and $D$ respectively, and $C_j\cdot D_j$ means the the standard inner product of $C_j$ and $D_j$.

By (21) and (22), we have
\begin{equation*}
\begin{split}
C_j=&(1,2,\cdots, j-1, \underbrace{j,\cdots, j}_{r+2-j})^{\textrm{T}},\\
D_j=&(\underbrace{n+2-j,n+2-j,\cdots, n+2-j}_{r-n+j}, n+1-j,n-j,\cdots, 1)^{\textrm{T}}.
\end{split}
\end{equation*}
For fixed $j\in\{1,2,\cdots, n+1\}$,
\begin{equation*}
\begin{split}
C_j\cdot D_j=&(n+2-j)(1+2+\cdots+(j-1))+j(1+2+\cdots+(n-j+1))\\
             &+j(n-j+2)(r+1-n)\\
            =&(2 -  \frac{ n^2}{2}+ 2  r+ n r)j + (\frac{ n}{2}-1-r)j^2.
\end{split}
\end{equation*}
It follows that
\begin{equation*}
\begin{split}
C^{m}_{k,r,n}&=(2 -  \frac{ n^2}{2}+ 2  r+ n r)\sum_{j=1}^{n+1}j + (\frac{ n}{2}-1-r)\sum_{j=1}^{n+1}j^2\\
             &=1+\frac{4 n}{3}+\frac{n^2}{12}-\frac{n^3}{3}-\frac{n^4}{12}+r+\frac{11 n r}{6}+n^2 r+\frac{n^3 r}{6}.
\end{split}
\end{equation*}
\end{proof}

It is not difficult to check that our formulas are exactly the same as those in Theorem 31 in \cite{BBS} for specially chosen $k$, $r$ and $n$.

\begin{example}
We compute the multiplicity $x^{40}_{4,8,8}$ of $V(4)\otimes V(8)\otimes V(8)$ in $V^{40}=S^{40}(\Bbb C^2\otimes\Bbb C^{2}\otimes\Bbb C^{2})$. From Theorem 3.3, one has
\begin{equation*}
\begin{split}
x_{4,8,8}^{40}=&dimV^{40}_{(4,8,8)}-dimV^{40}_{(6,8,8)}-dimV^{40}_{(4,10,8)}-dimV^{40}_{(4,8,10)}+dimV^{40}_{(6,10,8)}\\
                    &+dimV^{40}_{(6,8,10)}+dimV^{40}_{(4,10,10)}-dimV^{40}_{(6,10,10)}\\
               =&C^{40}_{18,16,16}-C^{40}_{17,16,16}-2C^{40}_{18,16,15}+2C^{40}_{17,16,15}+C^{40}_{18,15,15}-C^{40}_{17,15,15}.
\end{split}
\end{equation*}
See that the condition in Proposition 4.1. (III.1) holds when $k=18$, $r=16$ and $n=16$. Hence applying the formula, we have $C^{40}_{18,16,16}=6957$ . Similarly check that
\begin{equation*}
\begin{split}
C^{40}_{17,16,16}&=6710,\quad C^{40}_{18,16,15}=6421,\quad C^{40}_{17,16,15}=6208,\\
C^{40}_{18,15,15}&=5952,\quad C^{40}_{17,15,15}=5770.
\end{split}
\end{equation*}
Hence $x_{4,8,8}^{40}=3$.
\end{example}

\end{document}